\documentclass{amsart}
\usepackage[pagebackref]{hyperref}
\usepackage{amssymb}

\newcommand{\Q}{{\mathbb Q}}

\newcommand{\F}{{\mathbb F}}
\newcommand{\N}{{\mathbb N}}
\newcommand{\Z}{{\mathbb Z}}

\newcommand{\fp}{\mathfrak p}
\newcommand{\fP}{\mathfrak P}

\newcommand{\disc}{{\operatorname{disc}}}
\newcommand{\rank}{{\operatorname{rank}\,}}
\newcommand{\prank}{{\operatorname{rank}_p\,}}
\newcommand{\trank}{{\operatorname{rank}_2\,}}

\newcommand{\Cl}{{\operatorname{Cl}}}
\newcommand{\Ram}{{\operatorname{Ram}}}
\newcommand{\Gal}{{\operatorname{Gal}}}
\newcommand{\eps}{\varepsilon}

\newcommand{\lra}{\longrightarrow}
\newcommand{\too}{\longmapsto}

\newtheorem{thm}{Theorem}
\newtheorem{prop}[thm]{Proposition}
\newtheorem{lem}[thm]{Lemma}

\newtheorem{prob}[thm]{Problem}

\theoremstyle{definition}

\begin{document}
\title{The $4$-class group  of real quadratic number fields}
\author{Franz Lemmermeyer}
\address{M\"orikeweg 1, 73489 Jagstzell, Germany}
\email{hb3@ix.urz.uni-heidelberg.de}

\abstract 
In this paper we give an elementary proof of results on the
structure of $4$-class groups of real quadratic number fields
originally due to A.~Scholz. In a second (and independent)
section we strengthen C.~Maire's result that
the $2$-class field tower of a real quadratic number field is infinite
if its ideal class group has $4$-rank $\ge 4$, using a technique due
to F.~Hajir.  \endabstract

\maketitle

\section{Introduction}

Let $d$ and $d'$ be discriminants of real quadratic number fields,
and suppose that they are the product of positive prime
discriminants (equivalently, they are sums of two squares). 
If $(d,d') = 1$, and if $(d/p') = +1$ for all primes $p' \mid d'$,
then we can define a biquadratic Jacobi symbol
by $(d/d')_4 = \prod_{p' \mid d'} (d/p')_4$. Here
$(d/p')_4$ is the rational biquadratic residue symbol
(it is useful to put $(d/2)_4 = (-1)^{(d-1)/8}$; note that 
$d \equiv 1 \bmod 8$ if $8 \mid d'$). Observe, however,
that this symbol is not multiplicative in the numerator.
We also agree to say that a discriminant $d'$ divides
another discriminant $d$ if there exists a discriminant
$d''$ such that $d = d'd''$.

The following result due to R\'edei is well known:

\begin{prop}
Let $d$ be the discriminant of a quadratic number field.
The cyclic quartic extensions of $k$ which are unramified 
outside $\infty$ correspond to $C_4^+$-fac\-tor\-iza\-tions
of $d$, i.e. factorizations $d = d_1d_2$ into two
relatively prime positive discriminants such that
$(d_1/p_2) = (d_2/p_1) = +1$ for all $p_j \mid d_j$.
If $d = d_1d_2$  is such a $C_4^+$-fac\-tor\-ization, then the
extension $k(\sqrt{\alpha}\,)$ can be constructed by
choosing a suitable solution of $x^2 - d_1y^2 = d_2z^2$
and putting $\alpha = x+y\sqrt{d_1}$. In this case,
every cyclic extension of $k$ which is unramified outside 
$\infty$ and contains $\Q(\sqrt{d_1},\sqrt{d_2}\,)$ has
the form $k(\sqrt{d' \alpha}\,)$, where $d'$ is a 
discriminant dividing $d$.
\end{prop}

In particular, if $d$ is divisible by a negative prime
discriminant $d'$, then for every $C_4^+$-fac\-tor\-iza\-tion
$d = d_1d_2$ there is always a cyclic quartic extension
of $k = \Q(\sqrt{d}\,)$ which is unramified everywhere:
should  $k(\sqrt{\alpha}\,)$ be totally negative, simply
take $k(\sqrt{d' \alpha}\,)$. The problem of the existence
of cyclic quartic extension which are unramified everywhere
is thus reduced to the case where $d$ is a product of
positive prime discriminants. The next section is devoted
to an elementary proof of a result in this case; Scholz 
sketched a proof using class field theory in his
own special way in \cite{Sch}.

\section{Ramification at $\infty$}

Let $k$ be a real quadratic number field with discriminant
$d = \disc k$, and assume that $d$ is the product of
positive prime discriminants. Let $d = d_1\cdot d_2$ be
a $C_4^+$-factorization, i.e. assume that
$(d_1/p_2) = (d_2/p_1) = +1$ for all $p_1 \mid d_1$ and
all $p_2 \mid d_2$. Pick a solution of
$x^2-d_1y^2 = d_2z^2$ such that $k(\sqrt{d_1},\sqrt{\alpha}\,)$
(with $\alpha = x + y \sqrt{d_1}$) is a cyclic quartic extension 
of $k$ which is unramified outside $\infty$.

\begin{thm}\label{T1}
The extension $K = k(\sqrt{d_1},\sqrt{\alpha}\,)$ defined above 
is totally real if and only if $(d_1/d_2)_4 = (d_2/d_1)_4$.
If there is a cyclic octic extension $L/k$ containing $K$ and 
unramified outside $\infty$, then $(d_1/d_2)_4 = (d_2/d_1)_4 = +1$.
\end{thm}

\begin{proof}
Assume first that $d \equiv 1 \bmod 4$; then we can always
choose $x \in \Z$, $y \in \N$ in such a way that 
$x-1 \equiv y \equiv 1 \bmod 2$. Then $\alpha$ will be
$2$-primary if and only if $x+y \equiv 1 \bmod 4$.
We now proceed as in \cite{Lem} and reduce the equation
$x^2-d_1y^2 = d_2z^2$ modulo certain primes.
From $x^2 \equiv d_1y^2 \bmod d_2$ and
$ x^2 \equiv d_2z^2 \bmod d_1$ we get
$$ \Big(\frac{x}{d_2}\Big) = 
        \Big(\frac{d_1}{d_2}\Big)_4\Big(\frac{y}{d_2}\Big), \qquad
   \text{and } \qquad \Big(\frac{x}{d_1}\Big) = 
        \Big(\frac{d_2}{d_1}\Big)_4\Big(\frac{z}{d_1}\Big).$$
Multiplying these equations gives
$$ \Big(\frac{d_1}{d_2}\Big)_4 \Big(\frac{d_2}{d_1}\Big)_4 =
        \Big(\frac{x}{d_1d_2}\Big)\Big(\frac{y}{d_2}\Big)
        \Big(\frac{z}{d_1}\Big).$$
Considering $x^2-d_1y^2 = d_2z^2$ modulo $x$ shows
$$ \Big(\frac{-d_1}{x}\Big) = \Big(\frac{d_2}{x}\Big), \qquad
   \text{i.e.} \qquad \Big(\frac{x}{d_1d_2}\Big) = 
        \Big(\frac{d_1d_2}{x}\Big) = \Big(\frac{-1}{x}\Big).$$
Now write $y = 2^ju$ for some odd $u \in \N$; then 
$$ \Big(\frac{y}{d_2}\Big) = \Big(\frac{2}{d_2}\Big)^j
        \Big(\frac{u}{d_2}\Big)
        = \Big(\frac{2}{d_2}\Big)^j\Big(\frac{d_2}{u}\Big)
        = \Big(\frac{2}{d_2}\Big)^j.$$
But $j = 1$ implies $d_2 \equiv 5 \bmod 8$, and for $j \ge 2$ 
we have $d_2 \equiv 1 \bmod 8$; this shows that $(2/d_2)^j = (-1)^{y/2}$.
Finally we easily see that $(z/d_1) = (d_1/z) = 1$, hence we have
shown
$$ \Big(\frac{d_1}{d_2}\Big)_4 \Big(\frac{d_2}{d_1}\Big)_4 =
        \Big(\frac{-1}{x}\Big) (-1)^{y/2} =
        (-1)^{(|x|+y-1)/2}.$$
The right hand side equals $+1$ if $x > 0$ (the assumption
that $\alpha$ is $2$-primary says $x+y \equiv 1 \bmod 4$),
and is $-1$ if $x < 0$. Therefore $K$ is real if and only
if $(d_1/d_2)_4 (d_2/d_1)_4 = + 1$.

Now assume that $L/k$ is a cyclic unramified octic extension 
containing $K$. We claim that the prime ideals $\fp$ ramified
in $k_1(\sqrt{\alpha}\,)/k_1$ (where $k_j = \Q(\sqrt{d_j}\,)$)
must split in $k_1(\sqrt{\alpha'}\,)/k_1$, where $\alpha'$
is the conjugate of $\alpha$.

We first show that such $\fp$ are not ramified in 
$k_1(\sqrt{\alpha'}\,)/k_1$: in fact, ramifying primes
must divide $d_1d_2$, since $L/k$ is unramified. If
$\fp \mid d_1$, then $\fp$ ramifies completely in 
$k_1(\sqrt{\alpha}\,)/\Q$, which contradicts the fact
that all ramification indices must divide $2$ (again
because $L/k$ is unramified). Assume that an $\fp \mid d_2$
ramifies in both $k_1(\sqrt{\alpha}\,)/k_1$ and 
$k_1(\sqrt{\alpha'}\,)/k_1$; since primes dividing $d_2$ split 
in $k_1/\Q$ and ramify in $F/k_1$ (where $F = k_2k_1$),
$\fp$ would ramify in all three quadratic extensions of $k_1$
contained in $L$, and again its ramification index would 
have to be $\ge 4$.

Now suppose that $\fp$ is inert in $k_1(\sqrt{\alpha'}\,)/k_1$.
Then the prime ideal $\fP$ in $F$ above $\fp$ is inert in $K/F$, 
and since $L/F$ is cyclic, it is inert in $L/F$. Let $e$,
$f$ and $g$ denote the order of the ramification, inertia
and decomposition group $V$, $T$, and $Z$ of $\fP$, respectively; 
we have seen that $e = 2$, $f \ge 4$ and $g \ge 2$. Since 
$efg = (L:\Q) = 16$, we must have equality. Thus $k_1$ is the 
splitting field, and we have $Z = \Gal(L/k_1)$. We know that 
$\Gal(L/k_1) \simeq D_4$, and since $Z/T$ is always
cyclic of order $f$, $T$ must fix a cyclic quartic extension of
$k_1$ in $L$. But such an extension does not exist; this
contradiction shows that $\fp$ splits.

Finally, prime ideals $\fp$ ramifying in 
$k_1(\sqrt{\alpha}\,)/k_1$ divide $\alpha$ (otherwise 
$\fp$ would be a prime above $2$; but here we assume that 
$d$ is odd, i.e. $K/\Q$ is unramified above $2$); they 
split in $k_1(\sqrt{\alpha'}\,)/k_1$ if and only if 
$(\alpha'/\fp) = 1$. But now
$$ \Big(\frac{x-y\sqrt{d_1}}{\fp}\Big) =
   \Big(\frac{2x}{\fp}\Big) = \Big(\frac{2x}{p}\Big),$$
where $p$ is the prime under $\fp$; here we have used
that $x-y\sqrt{d_1} \equiv x-y\sqrt{d_1} + (x+y\sqrt{d_1})
\equiv 2x \bmod \fp$.
The proof above shows that $(2x/d_2) = (2/d_2)^j (d_1/d_2)_4$;
since $(2/d_2) = (2/d_2)^j$, we conclude that $\fp$ splits
in $k_1(\sqrt{\alpha}\,)/k_1$ only if $(d_1/d_2)_4 = 1$.

The proof in the case where one of the $d_i$ is divisible by 
$8$ is left to the reader.
\end{proof}

Theorem \ref{T1} contains many results on the solvability
of negative Pell equations (due to Dirichlet, Epstein, Kaplan
and others) as special cases. In fact, consider the case
$d = pqr$, where $p \equiv q \equiv r \equiv 1 \bmod 4$
are primes. Assume that $(p/q) = (p/r) = - (q/r) = 1$.
Then $d = p \cdot qr$ is the only $C_4^+$-factorization
of $d$, hence $\Cl_2^+(k) \simeq (2,2^n)$ for some
$n \ge 2$. If $(p/qr)_4 = -(qr/p)_4$, then there exists
a totally complex cyclic quartic extension unramified
outside $\infty$: this is only possible if $N\eps = +1$
for the fundamental unit $\eps$ of $k = \Q(\sqrt{d}\,)$.
If $(p/qr)_4 = (qr/p)_4$, on the other hand, then the
cyclic quartic extension is unramified everywhere, and
if $(p/qr)_4 = (qr/p)_4 = -1$, it cannot be extended to
a cyclic octic extension unramified outside $\infty$:
this shows that $\Cl_2(k) \simeq \Cl_2^+(k) \simeq (2,4)$
in this case, and in particular, $N\eps = -1$. 
Finding similar criteria or proving existing ones
(e.g. those in \cite{BLS3}) is no problem at all.

\section{Maire's Result}

In \cite{Mai96}, C. Maire showed that a real quadratic number
field $k$ has infinite $2$-class field tower if its class
group contains a subgroup of type $(4,4,4,4)$. Here we will
show that it suffices to assume that its class group in
the strict sense contains a subgroup of type $(4,4,4,4)$.
The method employed is taken from F. Hajir's paper \cite{Haj95}.

\begin{thm}\label{ThM}
Let $k$ be a real quadratic number field. If the strict ideal class
group of $k$ contains a subgroup of type $(4,4,4,4)$, then
the $2$-class field tower of $k$ is infinite.
\end{thm}

For the proof of this theorem we need a few results. For an
extension $K/k$, let the relative class group be defined
by $\Cl(K/k) = \ker(N_{K/k}: \Cl(K) \lra \Cl(k))$. Moreover,
let $G_p$ denote the $p$-Sylow subgroup of a finite abelian 
group $G$; we will denote the dimension of $G/G^p$ as an
$\F_p$-vector space by $\prank G$. Let $\Ram (K/k)$ denote 
the set of all primes in $k$ (including those at $\infty$) 
which ramify in $K/k$; we will also need the unit groups 
$E_k$ and $E_K$, as well as the subgroup 
$H = E_k \cap N_{K/k}K^\times$ of $E_k$.
Then Jehne \cite{Jeh79} has shown

\begin{prop}\label{PJ}
Let $K/k$ be a cyclic extension of prime degree $p$.
Then $$\rank \Cl_p(K/k) \ge \# \Ram (K/k) - \prank E_k/H - 1.$$
\end{prop}

Let us apply the inequality of Golod-Shafarevic to the
$p$-class group of $K$. We know that the $p$-class
field tower of $K$ is infinite if
$$\prank \Cl(K) \ge 2 + 2\sqrt{\prank E_K + 1}.$$
By Prop. \ref{PJ}, we know that
$$\prank \Cl(K) \ge \prank \Cl(K/k) \ge 
                \# \Ram (K/k) - \prank E_k/H - 1.$$
Thus $K$ has infinite $p$-class field tower if
$$\# \Ram (K/k) \ge 3 + \prank E_k/H + 2\sqrt{\prank E_K + 1}.$$
We have proved (compare Schoof \cite{Sch86}):

\begin{prop}\label{SchP}
Let $K/k$ be a cyclic extension of prime degree $p$, and let $\rho$
denote the number of finite and infinite primes ramifying in $K/k$.
Then the $p$-class field tower of $K$ is infinite if
$$ \rho \ \ge \ 3 + \prank E_k/H + 2\sqrt{\prank E_K + 1}.$$
Here $H$ is the subgroup of $E_k$ consisting of units which are
norms of elements from $K$, and $\prank G$ denotes the $p$-rank
of $G/G^p$.  
\end{prop}

\noindent{\em Proof of Thm. \ref{ThM}}.
Since the claim follows directly from the inequality of
Golod-Shafarevic if $\rank \, \Cl_2(k) \ge 6$, we may
assume that $d = \disc k$ is the product of at most
six positive prime discriminants, or of at most seven 
if one of them is negative. The idea is to show that 
a subfield of the genus class field of $k$ satisfies
the inequality of Prop. \ref{SchP}: this will clearly
prove that $k$ has infinite $2$-class field tower.

Assume first that $d = \prod_{j=1}^5 d_j$ is the product 
of five positive prime discriminants. If $\Cl^+(k)$ contains
a subgroup of type $(4,4,4,4)$, then $d = d_1 \cdot d_2d_3d_4d_5$
and $d = d_2 \cdot d_1d_3d_4d_5$
must be $C_4^+$-factorizations. Therefore, the $d_j$ ($j \ge 3$)
split completely in $F = \Q(\sqrt{d_1},\sqrt{d_2}\,)$, hence 
there are 12 prime ideals ramifying in $K/F$, where 
$K = F(\sqrt{d_3d_4d_5}\,)$.
Since $\trank E_K = 8$, the condition of Prop. \ref{SchP}
is satisfied for $K/F$ if $\trank E_F/H \le 3$. But
since $-1$ is a norm in $K/F$ (only prime ideals of norm
$\equiv 1 \bmod 4$ ramify), we have $-1 \in H$; this implies
that $\trank E_F/H \le 3$.

Next suppose that $d = \prod_{j=1}^6 d_j$ is the product of 
six positive prime discriminants. We will show in the next 
section (see Problem \ref{P1}) that either there is 
$C_4^+$-factorizations of type $d_1 \cdot d_2d_3d_4d_5d_6$ 
(then we can apply Prop. \ref{SchP} to the quadratic extension 
$\Q(\sqrt{d},\sqrt{d_1}\,)/\Q(\sqrt{d_1}\,)$, and we are done), 
or there exist (possibly after a suitable permutation of the 
indices) $C_4^+$-factorizations $d_1d_2 \cdot d_3d_4d_5d_6$ 
and $d_1d_3 \cdot d_2d_4d_5d_6$.  In this case consider 
$F = \Q(\sqrt{d_1d_2},\sqrt{d_1d_3}\,)$. The prime ideals 
above $d_4, d_5, d_6$ split completely in $F/\Q$, and 
applying Prop. \ref{SchP} to $F(\sqrt{d}\,)/F$ 
shows exactly as above that $F(\sqrt{d}\,)$ has
infinite $2$-class field tower.

Now we treat the case where $d$ is divisible by some 
negative prime discriminants. First assume that $d$ is 
the product of six prime discriminants. In this case,
all possible factorizations of $d$ as a product of
two positive discriminants must be a $C_4^+$-factorization;
unless $d$ is a product of six negative discriminants,
there exists a factorization of type $d_1 \cdot \prod_{j = 2}^6 d_j$,
where  $d_1 > 0$ is a prime discriminant. In this case,
we can apply Prop. \ref{SchP} to
$\Q(\sqrt{d},\sqrt{d_1}\,)/\Q(\sqrt{d_1}\,)$. 
The case where
all six prime discriminants are negative does not occur here:
since all factorizations are $C_4^+$, we must have
$(d_id_j/q_\ell) = + 1$ for all triples $(i,j,\ell)$
such that $1 \le i < j < \ell \le 6$ (here $q_\ell$
is the unique prime dividing $d_\ell$). At least one
of these triples consist of odd primes $q \equiv 3 \bmod 4$;
call them $q_1, q_2$ and $q_3$, respectively.
We find $(q_1/q_3) = (q_2/q_3)
= - (q_3/q_2) = - (q_1/q_2) = (q_2/q_1) = (q_3/q_1)$,
contradicting the quadratic reciprocity law.

Finally, if $d$ is the product of seven prime discriminants,
then (see Problem \ref{P2}) there exists at least one 
$C_4^+$-factorization of type $d = p \cdot **$ or 
$d = rr' \cdot **$, where $p \equiv 1 \bmod 4$ and 
$r, r'$ are either both positive or both negative prime 
discriminants (see Problem \ref{P2}). Applying Prop. \ref{SchP} to 
$\Q(\sqrt{d},\sqrt{p}\,)/\Q(\sqrt{p}\,)$ or
$\Q(\sqrt{d},\sqrt{rr'}\,)/\Q(\sqrt{rr'}\,)$
yields the desired result.
\qed

Observe that the same remarks as in Hajir's paper apply: 
our proof yields more than we claimed. If, for example,
$d$ is a product of five positive prime discriminants and
admits two $C_4^+$-factorizations of type
$d = d_1\cdot d_2d_3d_4d_5$ and $d = d_2\cdot d_1d_3d_4d_5$
then $k = \Q(\sqrt{d}\,)$ has infinite $2$-class field tower
even if $\Cl(k)$ has $4$-rank equal to $2$. 

We also remark that it is not known whether these results are
best possible: we do not know any imaginary quadratic number field
with $2$-class group of type $(4,4)$ and finite $2$-class field
tower, and similarly, there is no example in the real case 
with $\Cl_2^+(k) \simeq (4,4,4)$. There are, however, real
quadratic fields with  $\Cl_2(k) \simeq (4,4)$ and abelian
$2$-class field tower (cf. \cite{BLS2}).

\section{Some Ramsey-type Problems}

Let the discriminant $d = d_1 \ldots d_t$ be the product of
$t$ positive prime discriminants. Let $X$ be the subspace
of $\F_2^t$ consisting of $0$ and $(1, \ldots, 1)$,
and put $V = \F_2^t/X$. Then the set of possible
factorizations into a product of two discriminants
corresponds bijectively to an element of $V$: in fact,
a factorization of $d$ is a product 
$d = \prod_j d_j^{e_j} \cdot \prod_j d_j^{f_j}$
with $e_j, f_j \in \F_2$ and $e_j + f_j = 1$, and it
corresponds to the image of $(e_1, \ldots, e_t)$ in the
factor space $V = \F_2^t/X$ (exchanging the two factors
corresponds to adding $(1, \ldots, 1)$). We will always
choose representatives with a minimal number of nonzero
coordinates. Moreover, the
product defined on the set of factorizations of $d$ 
corresponds to the addition of the vectors in $V$: the
bijection constructed above is a group homomorphism.

Define a map $\F_2^t \lra \N$ by mapping $(e_1, \ldots, e_t)$ 
to $\min \{ \sum_{j=1}^t e_j, t - \sum_{j=1}^t e_j \}$;
observe that the sum is formed in $\N$ (not in $\F_2$).
Since $(1,\ldots,1) \too 0$, this induces a map $V \lra N$
which we will denote by $S$. Let $V_\nu$ denote the fibers
of $V$ over $\nu$, i.e. put $V_\nu = \{ u \in V: S(u) = \nu \}$.
It is easy to determine their cardinality:

\begin{lem}
If $t = 2s$ is even, then
$ \# V_\nu = \begin{cases}
        \binom{t}{\nu}          & \text{if } \nu < s \\
        \frac12 \binom{t}{s}    & \text{if } \nu = s \end{cases}; $
if $t = 2s+1$ is odd, then $ \# V_\nu = \binom{t}{\nu} $ for
all $\nu \le s$.
\end{lem}

Now let us formulate our first problem:

\begin{prob}\label{P1}
Let $t = 6$, and suppose that $U \subseteq V$ is a subspace of
dimension $4$. If $U \cap V_1$ is empty, then the equation
$a+b+c = 0$ has solutions in $U \cap V_2$.
\end{prob}

\begin{proof}
Clearly $\# U= 16$ and $\# (U \cap V_0) = 1$; if 
$U \cap V_1 = \varnothing$,
then $ \# (U \cap V_3) \le \# V_3 = 10$ implies
that $ \# (U \cap V_2) \ge 4$. But among four vectors in
$V_2$ there must exist a pair $a,b \in U \cap V_2$ with a 
common coordinate $1$ by Dirichlet's box principle; after 
permuting the indices if necessary we may assume that 
$a = (1,1,0,0,0,0)$ and $b = (1,0,1,0,0,0)$. Clearly
$a+b = (0,1,1,0,0,0) \in U \cap V_2$, and our claim is proved.
\end{proof}

What has this got to do with our $C_4^+$-factorizations?
Well, consider the case where $d$ is the product of six 
positive prime discriminants. The possible 
$C_4^+$-factorizations correspond to $V$; since 
$\rank \Cl_4^+(k) \ge 4$, we must have at least 
four independent $C_4^+$-factorizations, generating
a subspace $U \subset V$ of dimension $4$. Problem
\ref{P1} shows that among these there is one
$C_4^+$-factorization with one factor a prime discriminant,
or there are two $C_4^+$-factorizations 
$d_1d_2\cdot d_3d_4d_5d_6$ and $d_1d_3\cdot d_2d_4d_5d_6$. 

Now consider $\F_2^t/X$ and define 'incomplete traces'
$$T_{2k}: \F_2^t/X \lra \F_2: (u_1, \ldots, u_t) \too
         \sum_{j=1}^{2k} u_j$$
(here the sum {\em is} formed in $\F_2$). This is a well
defined linear map, hence its kernel $V^{(2k)} = \ker T_{2k}$ 
is an $\F_2$-vector space of dimension $t-2$. 
The elements of $V^{(2k)}$ correspond to the set of
factorizations of a discriminant $d$ into a product
of two positive discriminants when $2k$ of the
$t$ prime discriminants dividing $d$ are negative.
Defining $S: V^{(2k)} \lra \N$ as above and denoting
the fibers by $V^{(2k)}_\nu$, we find, for example,
that $\# V^{(2k)}_0 = 1$ and $V^{(2k)}_1 = t-2k$.
In the special case $t = 7$ we get, in addition,
$\# V^{(2)}_2 = 1 +  \binom52 = 11$, 
$\# V^{(2)}_3 = 5 +  \binom52 = 15$, 
$\# V^{(4)}_2 = 3 +  \binom42 = 9$, 
$\# V^{(4)}_3 = 1 +  3\binom52 = 19$, 
$\# V^{(6)}_2 = \binom62 = 15$, and 
$\# V^{(6)}_3 = \binom62 = 15$.

\begin{prob}\label{P2}
Let $t = 7$, $1 \le k \le 3$, and consider a $4$-dimensional 
subspace $U$ of $V^{(2k)}$. Then at least one of 
$U \cap V^{(2k)}_1$ or $U \cap V^{(2k)}_2$ is not empty.
\end{prob}

\begin{proof}
Assume that $U \cap V^{(2k)}_1 = U \cap V^{(2k)}_2 = \varnothing$.
Then $U \subseteq (V^{(2k)}_0 \cup V^{(2k)}_3)$.
If $k = 2$ or $k = 6$, this leads at once to a contradiction,
because then $\# V^{(2k)}_3 = 15 = \# U - 1$, and
$V^{(2k)}_3 \cup V^{(2k)}_0$ is not a subspace in these cases
(for example, $(0,0,1,1,0,0,1)+(0,0,1,0,1,0,1) = (0,0,0,1,1,0,0)
\in V^{(2k)}_2$ for $k = 2$ and $k = 6$). Consider the case 
$k = 4$; since
$U \cap V^{(2k)}_3$ contains $15$ elements, there exists
$u \in U \cap V^{(2k)}_3 \setminus \{(0,0,0,0,1,1,1)\}$.
Permuting the indices if necessary (of course we are not allowed
to exchange one of the first $2k$ indices with one from the
last $t-2k$) we may assume that
$u = (1,1,0,0,0,0,1)$. It is easy to see that $V^{(2k)}_3$
contains exactly six elements $v$ such that $u+v \in V^{(2k)}_2$:
since $\# V^{(4)}_3 = 19$ and $\# U \cap V^{(2k)}_3 
\setminus \{u,(0,0,0,0,1,1,1)\} \ge 14$, one of these $v$ 
must be contained in $U\cap V^{(2k)}_3$. This shows that
$V^{(2k)}_3 \cup V^{(2k)}_0$ does not contain a subspace of 
dimension $4$, and the proof is complete.
\end{proof}

\end{document}